\documentclass[a4paper,12pt,reqno]{amsart}
\usepackage{lmodern,amsmath,amssymb,mathtools,amsthm}
\usepackage{enumerate,cite}
\usepackage{graphics,graphicx,geometry}
\usepackage{hyperref}  
\usepackage{mathrsfs} 
\usepackage{wasysym}
\usepackage{cleveref}
\usepackage{floatrow,morefloats}
\usepackage{epsfig}
\usepackage{epstopdf}
\usepackage{caption,subcaption,xcolor}
\usepackage{longtable}
\usepackage{array}
\usepackage{multirow}
\usepackage{hhline}
\newcolumntype{M}[1]{>{\centering\arraybackslash}m{#1}}
\newtheorem{definition}{Definition}
\newtheorem{lemma}{Lemma}
\newtheorem{theorem}{Theorem}
\newtheorem{corollary}{Corollary}

\numberwithin{equation}{section}
\numberwithin{definition}{section}
\numberwithin{lemma}{section}
\numberwithin{theorem}{section}
\numberwithin{corollary}{section}
\numberwithin{example}{section}

\title{Sufficiency for Nephroid Starlikeness using Hypergeometric Functions}
\author{A. Swaminathan$^\ddagger$}
   \address{$^\ddagger$Department of Mathematics\\ Indian Institute of Technology, Roorkee-247667, Uttarakhand, India}
   \email{mathswami@gmail.com, a.swaminathan@ma.iitr.ac.in}
\author{Lateef Ahmad Wani$^\dagger$}
\address{$^\dagger$Department of Mathematics\\ Indian Institute of Technology, Roorkee-247667, Uttarakhand, India}
\email{lateef17304@gmail.com}
\allowdisplaybreaks
\bigskip
\begin{document}
%%%================
	
%%%%%%%%%%%%%%%%%%%% ABSTRACT %%%%%%%%%%%%%%%%%%%%%%%%
\begin{abstract}
Let $\mathcal{A}$ consists of analytic functions $f:\mathbb{D}\to\mathbb{C}$ satisfying $f(0)=f'(0)-1=0$. Let $\mathcal{S}^*_{Ne}$ be the recently introduced Ma-Minda type functions family associated with the $2$-cusped kidney-shaped {\it nephroid} curve $\left((u-1)^2+v^2-\frac{4}{9}\right)^3-\frac{4 v^2}{3}=0$ given by
\begin{align*}
	\mathcal{S}^*_{Ne}:=
	\left\{f\in\mathcal{A}:\frac{zf'(z)}{f(z)}\prec\varphi_{\scriptscriptstyle {Ne}}(z)=1+z-z^3/3\right\}.
\end{align*}
In this paper, we adopt a novel technique that uses the geometric properties of {\it hypergeometric functions} to determine sharp estimates on $\beta$ so that each of the differential subordinations 
\begin{align*}
	p(z)+\beta zp'(z)\prec
	\begin{cases}
		\sqrt{1+z};\\
		1+z;\\
		e^z;
	\end{cases}
\end{align*}
imply $p(z)\prec\varphi_{\scriptscriptstyle{Ne}}(z)$, where $p(z)$ is analytic satisfying $p(0)=1$. As applications, we establish conditions that are sufficient to deduce that $f\in\mathcal{A}$ is a member of $\mathcal{S}^*_{Ne}$.  	
\end{abstract}
%%%%%%%%%%%%%%%%%%%%%%%%%%%%%%%%%%%%%%%%%%%%%%%%%%%%%%%
%%%----------------
\maketitle
%%%----------------
\markboth{A. Swaminathan and Lateef Ahmad Wani}{Sufficiency for nephroid starlikeness using hypergeometric functions}
%%%==========================================%%%

%\begin{center}
%Dedicated to Professor H. M. Srivastava on the occasion of his $80$-th birthday
%\end{center}

%%%%%%%%%%%%% SECTION INTRODUCTION %%%%%%%%%%%%%%%%%%%%%%%
%%%%%%%%%%%%%%%%%%%%%%%%%%%%%%%%%%%%%%%%%%%%%
\section{Introduction}
Let $\mathcal{A}$ be the family of analytic functions $f$ defined on the open unit disk $\mathbb{D}:=\left\{z:|z|<1\right\}$ and satisfying $f(0)=f'(0)-1=0$. Let $\mathcal{S}\subset\mathcal{A}$ be the family of one-one ({\it univalent}) functions defined on $\mathbb{D}$. Further, let $\mathcal{S}^*\subset\mathcal{S}$ and $\mathcal{C}\subset\mathcal{S}$ be, respectively, the well-known classes of {\it starlike} and {\it convex} functions defined on $\mathbb{D}$. We note that the functions in $\mathcal{S}^*$ are analytically characterized by the condition that for each $z\in\mathbb{D}$, the quantity $zf'(z)/f(z)$ lies in the interior of the half-plane $\mathrm{Re}(w)>0$.
\par 
Let $f$ be analytic and $g$ be univalent. Then $f$ is {\it subordinate} to $g$, written as $f\prec{g}$, if, and only if, 
\begin{align*}
	f(0)=g(0) \quad \text{and} \quad f(\mathbb{D})\subset g(\mathbb{D}).
\end{align*}
\begin{definition}
	Let $\Lambda:\mathbb{C}^2\times\mathbb{D}\to\mathbb{C}$ be analytic, and let $u$ be univalent. The analytic function $p$ is said to satisfy the first-order differential subordination if
	\begin{align}\label{Def-Diff-Subord-Psi}
		\Lambda(p(z),\,zp'(z);\,z)\prec u(z), \qquad z\in\mathbb{D}.
	\end{align}
\end{definition}
If $q:\mathbb{D}\to\mathbb{C}$ is univalent and $p{\prec}q$ for all $p$ satisfying \eqref{Def-Diff-Subord-Psi}, then $q$ is said to be a dominant of the differential subordination \eqref{Def-Diff-Subord-Psi}. A dominant $\tilde{q}$ that satisfies $\tilde{q}\prec{q}$ for all dominants $q$ of \eqref{Def-Diff-Subord-Psi} is called the best dominant of \eqref{Def-Diff-Subord-Psi}. If $\tilde{q}_1$ and $\tilde{q}_2$ are two best dominants of \eqref{Def-Diff-Subord-Psi}, then $\tilde{q}_2(z)=\tilde{q}_1(e^{i\theta}z)$ for some $\theta\in\mathbb{R}$.
For further details related to differential subordinations, we refer to the monograph of Miller and Mocanu \cite{Miller-Mocanu-Book-2000-Diff-Sub} (see also Bulboac\v{a} \cite{Bulboaca-2005-Diff-Sub-Book}).  
Due to its straightforward consequences, the theory of differential subordinations (a complex analogue of differential inequalities) developed by Miller and Mocanu \cite{Miller-Mocanu-Book-2000-Diff-Sub} is being extensively used in studying the analytic and geometric properties of univalent functions.
For some recent works, see 
\cite{Antonio-Miller-2020-AMP,
	Ebadian-Bulboaca-Cho-RACSAM-2020,
	Ebadian-Adegani-Bulboaca-2020-JFS,
	Gavris-2020-Slovaca,
	S.Kumar-Goel-2020-RACSAM,
	Naz-Ravi-2020-MJM,
   Swami-Wani-2020-BKMS, 
HMS-DS-Caratheo-2020-JIA}. 
%%%============================
\par 
\bigskip
%%%%%============================
Following  
\cite{Ali-Jain-Ravi-2012-Radii-LemB-AMC, 
	Sokol-J.Stankwz-1996-Lem-of-Ber,
%	Sokol-1998-Cassinian-Curve,
	Mendiratta-2014-Shifted-Lemn-Bernoulli,
%	Raina-Sokol-2015-Crescent-Shaped-I,
	Gandhi-Ravi-2017-Lune,
%	Raina-Sharma-Sokol-2018-Crescent-Shaped,
	Sharma-Raina-Sokol-2019-Ma-Minda-Crescent-Shaped,
    Mendiratta-Ravi-2015-Expo-BMMS,
    Sharma-Ravi-2016-Cardioid,
    Kumar-Ravi-2016-Starlike-Associated-Rational-Function,
    Kargar-2019-Booth-Lem-A.M.Physics,
    Cho-2019-Sine-BIMS,
    Khatter-Ravi-2019-Lem-Exp-Alpha-RACSAM,
    Goel-Siva-2019-Sigmoid-BMMS, 
    Yunus-2018-Limacon},
the authors in \cite{Wani-Swami-Nephroid-Basic,Wani-Swami-Radius-Problems-Nephroid-RACSAM} introduced and studied the geometric properties of the function $\varphi_{\scriptscriptstyle{Ne}}(z):=1+z-z^3/3$ and the associated Ma-Minda type 
(see \cite{Ma-Minda-1992-A-unified-treatment, 
	HMS-MM-2018-RACSAM, 
	HMS-MM-2013-JCA})
 function family $\mathcal{S}^*_{Ne}$ given by
\begin{align*}
	\mathcal{S}^*_{Ne}:=\left\{f\in\mathcal{A}:\frac{zf'(z)}{f(z)}\prec\varphi_{\scriptscriptstyle {Ne}}(z)\right\}.
\end{align*}
%%%%%%%%%%%%%%%%%%%%%%%%%%
It was proved by Wani and Swaminathan \cite{Wani-Swami-Nephroid-Basic} that the function $\varphi_{\scriptscriptstyle{Ne}}(z)$ maps the boundary  $\partial\mathbb{D}$ of the unit disk $\mathbb{D}$ univalently onto the {\it nephroid}, a $2$--cusped kidney--shaped curve (see \Cref{Figure-Nephroid}), given by
\begin{align}\label{Equation-of-Nephroid}
	\left((u-1)^2+v^2-\frac{4}{9}\right)^3-\frac{4 v^2}{3}=0.
\end{align}
Geometrically, a nephroid is the locus of a point fixed on the circumference of a circle of radius $\rho$ that rolls (without slipping) on the outside of a fixed circle having radius $2\rho$.
First studied by Huygens and Tschirnhausen in 1697, the nephroid curve was shown to be the catacaustic (envelope of rays emanating from a specified point) of a circle when the light source is at infinity. In 1692, Jakob Bernoulli had shown that the nephroid is the catacaustic of a cardioid for a luminous cusp. However, the word nephroid was first used by Richard A. Proctor in 1878 in his book `{\it The Geometry of Cycloids}'. For further details related to the nephroid curve, we refer to \cite{Lockwood-Book-of-Curves-2007, Yates-1947-Handbook-Curves}.
%%%%%%%%%%%%%%%%%%%%%%%%%%%%%%%%%%%% FIGURE NEPHROID %%%%%%%%%%%%%%%%%%%%%%%%%%%%%%%%%%%%%%%%%%%%%%%%%%%%
\begin{figure}[H]
	\includegraphics[scale=1]{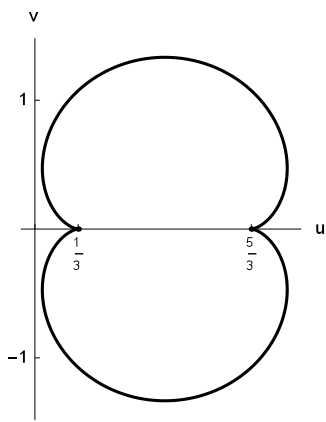}
	\caption{Nephroid: The Boundary curve of
		$\varphi_{\scriptscriptstyle{Ne}}(\mathbb{D})$.}
	\label{Figure-Nephroid}	
\end{figure}
%%%%%%%%%%%%%%%%%%%%%%%%%%%%%%%%%%%%%%%%%%%%%%%%%%%%%%%%%%%%%%%%%%%%%%%%%%%%%%%%%%%%%%%%%%%%%%%%%%%%%%%%
%%%%%%%%%%%%%%%%%%%%%%%%%%%%%%%%%%%%%%%%%%%%%%%%%%%%%%%%%%%%%%%%%%%%%%%%%%%%%%%%%%%%%%%%%%%%%%%%%%%%%%%%%%%%%%%%%%%%%%%%%%%%%%%%
Thus $f\in\mathcal{S}^*_{Ne}$ if, and only if, all the values taken by the expression ${zf'(z)}/{f(z)}$ lie in the region $\Omega_{Ne}$ bounded by the nephroid curve \eqref{Equation-of-Nephroid}. Since $\mathcal{S}^*_{Ne}\subset\mathcal{S}^*$, we call $f\in\mathcal{S}^*_{Ne}$ a {\it nephroid starlike} function.
\par
\bigskip
In this paper, we employ the differential subordination techniques and use the geometric properties of {\it Gaussian and confluent hypergeometric functions} to establish conditions which ensure that the analytic function $f\in\mathcal{A}$ is nephroid starlike in $\mathbb{D}$. More specifically, we determine the best possible bounds on the real $\beta$ so that, for some analytic $p$ satisfying $p(0)=1$, the following implication holds:
\begin{align*}
	p(z)+\beta zp'(z)\prec
	\begin{cases}
		\sqrt{1+z};\\
		1+z;\\
		e^z;
	\end{cases}
\implies
p(z)\prec\varphi_{\scriptscriptstyle{Ne}}(z).
\end{align*}
Replacing $p(z)$ by the expression ${zf'(z)}/{f(z)}$ for any $f\in\mathcal{A}$, we obtain conditions that are sufficient to imply that the function $f$ is nephroid starlike in $\mathbb{D}$.
\par 
Although similar type of differential subordination implication problems have been studied  
for several other function families  
(for instance see 
\cite{Ahuja-Ravi-2018-App-Diff-Sub-Stud-Babe-Bolyai,
	Ali-Ravi-2012-Diff-Sub-LoB-TaiwanJM,
	Aktas-LemExpo-Coulomb-2020-SSMH,
	Kumar-Ravi-2013-Suff-Conditions-LoB-JIA, 
	SushilKumar-Ravi-2018-Sub-Positive-RP-CAOT, 
	Bohra-Ravi-2019-Diff-Subord-Hacet.J, 
	Cho-Ravi-2018-Diff-Sub-Booth-Lem-TJM, 
	Ebadian-Bulboaca-Cho-RACSAM-2020, 
	Madaan-Ravi-2019-Filomat,
	NazAdiba-Ravi-2019-Exponential-TJM, 
	Naz-Ravi-2020-MJM,
	S.Kumar-Goel-2020-RACSAM}), the approach of utilizing the properties of hypergeometric functions to arrive at the desired implication is totally new. In addition, this paper verifies {\it analytically} certain crucial facts which some of the above cited authors have concluded geometrically without providing any analytic clarification. However, graphical illustrations are also provided in this manuscript for enhancing the clarity of the results to the reader and competing with the existing related literature.
	\par 
 In the sequel, it is always assumed that $z\in\mathbb{D}$ unless stated otherwise.
%%%%%%%%%%%%%%%%%%%%%%%%%%%
\section{Preliminaries on Hypergeometric Functions}
The following lemma related to differential subordination will be used in our discussion.
\begin{lemma}[{Ma and Minda \cite[p. 132]{Miller-Mocanu-Book-2000-Diff-Sub}}]\label{Lemma-3.4h-p132-Miller-Mocanu}
	Let $q:\mathbb{D}\to\mathbb{C}$ be univalent, and let $\lambda$ and $\vartheta$ be analytic in a domain $\Omega\supseteq q(\mathbb{D})$ with $\lambda(\xi)\neq0$ whenever $\xi\in{q(\mathbb{D})}$. Define
	\begin{align*}
		\Theta(z):=zq'(z)\,\lambda(q(z)) \quad \text{ and } \quad  h(z):=\vartheta(q(z))+\Theta(z), \qquad z\in\mathbb{D}.
	\end{align*}
	Suppose that either
	\begin{enumerate}[\rm(i)]
		\item $h(z)$ is convex, or
		\item $\Theta(z)$ is starlike.\\
		In addition, assume that
		\item $\mathrm{Re}\left({zh'(z)}/{\Theta(z)}\right)>0$ in $\mathbb{D}$.
	\end{enumerate}
	If $p\in\mathcal{H}$ with $p(0)=q(0)$, $p(\mathbb{D})\subset{\Omega}$ and
	\begin{align*}
		\vartheta(p(z))+zp'(z)\,\lambda(p(z))\prec\vartheta(q(z))+zq'(z)\,\lambda(q(z)), \qquad z\in\mathbb{D},
	\end{align*}
	then $p\prec{q}$, and $q$ is the best dominant.
\end{lemma}
\begin{definition}[{\bf Gaussian hypergeometric function}]
	Let $a,b\in\mathbb{C}$ and $c\in\mathbb{C}\setminus\{0,-1,-2,\ldots\}$. Define
	\begin{align}\label{Gaussian-HG}
		F(a,b;c;z):={_2F_1}(a,b;c;z)=\sum_{j=0}^\infty\frac{(a)_j(b)_j}{j!\;(c)_j}\,z^j, \quad z\in\mathbb{D},
	\end{align}
	where $(x)_j$ is the Pochhammer symbol given by
	\begin{align}\label{Pochhammer-Symbol}
		(x)_j=
		\begin{cases}
			1, \quad j=0\\
			x(x+1)(x+2)\cdots(x+j-1), \qquad j\in\{1,2,\ldots\}.
		\end{cases}
	\end{align}
	The analytic function $F(a,b;c;z)$ given in \eqref{Gaussian-HG} is called the  Gaussian hypergeometric function.
\end{definition}
%%%%%%%%%%%%%%%%%%%%%%%%%%%%%%%%
Prior to the use of hypergeometric functions in the proof of Bieberbach's conjecture by de Branges \cite{de-Branges-1985-Proof-BC}, there has been little known connections between the univalent function theory and the theory of special functions. This surprising use of hypergeometric functions has given function theorists a renewed interest to study the interrelatedness of these two concepts and, as a result, a number of papers have been published in this direction. For instance, see
\cite{Ahuja-2008-Connections-HGFs-AMC, 
	Bohra-Ravi-CHGfs-Bessel-2017-AM,
	Kustner-2002-HGF-CMFT, 
	Kustner-2007-HGF-JMAA, 	
	Miller-Mocanu-1990-HGFs-PAMS,
	Mostafa-2010-HGFs-CMA,
	Ruscheweyh-Singh-1986-HGFs-JMAA,
	Swaminathan-2004-HGFs-TamsuiOxf,  
	Swaminathan-2006-HGFs-Conic-Regions, 
	Swaminathan-2006-Inclusion-HGFs-JCAA, 
	Swaminathan-2007-IncBeta-ITSF,
	Swaminathan-2010-HGFs-CMA, 
    Wani-Swami-2020-BabesBolyai,
HMS-HGF-2017-Math-Methods-Appl-Sci,
HMS-HGF-2019-RACSAM,   
HMS-HGF-2020-JNCA,
HMS-HGF-2021-Quaest}. 
%%%================================
%%%================================
The function $F(a,b;c;z)$ defined in \eqref{Gaussian-HG} has many interesting properties among which the following will be used to prove our results. For further details, we refer to Rainville \cite{Rainville-Special-Functions-Book}.  
\begin{enumerate}[(i)]
	\item $F(a,b;c;z)$ is a solution of the differential equation
	\begin{align*}
		z(1-z)w''(z)+(c-(a+b+1)z)w'(z)-abw(z)=0.
	\end{align*}	
	\item $F(a,b;c;z)$ has a representation in terms of the gamma function
	\begin{align*}
		\Gamma(z)=\int_0^\infty t^{z-1}e^{-t}dt,\quad \mathrm{Re}(z)>0
	\end{align*} 
	as
	\begin{align}\label{Gamma-Function-Rep-Gaussian-HGF}
		F(a,b;c;z)=\frac{\Gamma(c)}{\Gamma(a)\Gamma(b)}\sum_{j=0}^\infty\frac{\Gamma(a+j)\Gamma(b+j)}{j!\;\Gamma(c+j)}\,z^j.
	\end{align}
	\item $F(a,b;c;z)$ satisfies
	\begin{align}\label{Derivative-Property-GHGF}
		F'(a,b;c;z)=\frac{ab}{c}F(a+1,b+1;c+1;z)
	\end{align}
	\item If $\mathrm{Re}\,c>\mathrm{Re}\,b>0$, then $F(a,b;c;z)$ has the following integral representation
	\begin{align}\label{Integral-Rep-Gaussian-HGF}
		F(a,b;c;z)=\frac{\Gamma(c)}{\Gamma(b)\Gamma(c-b)}\int_{0}^1\frac{t^{b-1}(1-t)^{c-b-1}}{(1-tz)^{a}}\,dt, \quad z\in\mathbb{D}.
	\end{align}
\end{enumerate}
We hereby mention that the function $zF(a,b;c;z)$ given by
\begin{align*}
	zF(a,b;c;z)=z{_2F_1}(a,b;c;z)=z+\sum_{j=2}^\infty\frac{(a)_{j-1}(b)_{j-1}}{(j-1)!\;(c)_{j-1}}\,z^j, \quad z\in\mathbb{D},
\end{align*}
is known as {\it normalized} or {\it shifted} Gaussian hypergeometric function.
%%%%%%%%%%%%%%%%%%%%%%%%%%%%%%%%%%%%%%%%%%%%%%%%%
\subsection*{Order of Starlikeness}
%%%%%%%%%%%%%%%%%%%%%%%%%%%%%%%%%%%%%%%%%%%%%%%%
Let $f\in\mathcal{A}$. The order of starlikeness (with respect to zero) of the function $f(z)$ is defined to be the number $\sigma(f)$ given by
\begin{align}\label{Order-ST}
	\sigma(f):=
	\inf_{z\in\mathbb{D}}\mathrm{Re}\left(\frac{zf'(z)}{f(z)}\right)\in[-\infty,1].
\end{align}
In terms of $\sigma(f)$, we observe that $f\in\mathcal{A}$ is starlike if, and ony if, $\sigma(f)\geq0$, or precisely,
\begin{align*}
	f\in\mathcal{S}^* \iff  \sigma(f) \geq 0.
\end{align*}
%%%%%%%%%%%%%%%%%%%%%%%%%%%%%%%%%%%%%%%%%%
Related to the order of starlikeness of the modified Gaussian hypergeometric function $zF(a,b;c;z)$, K\"{u}stner
\cite{Kustner-2002-HGF-CMFT, Kustner-2007-HGF-JMAA} 
proved the following result.
%%%%%%%%%%%%%% Kustner Lemma %%%%%%%%%%%%%%%%%%
\begin{lemma}[{K\"{u}stner \cite[Theorem 1 (a)]{Kustner-2007-HGF-JMAA}}]\label{Lemma-Kustner-2007-HGF-JMAA}
	If $0<a\leq b\leq c$, then 
	\begin{align*}
		1-\frac{ab}{b+c} 
		\leq \sigma\left(zF(a,b;c;z)\right)
		\leq 1-\frac{ab}{2c}.
	\end{align*}	
\end{lemma}
%%%%%%%%%%%%%%%%%%%%%%%%%%%%%%%%%%
%%%%%%%%%%=====================================%%%%%%%%%%%%
\begin{definition}[{\bf Confluent hypergeometric function}]
	Let $a\in\mathbb{C}$ and $c\in\mathbb{C}\setminus\{0,-1,-2,\ldots\}$. The confluent (or Kummer) hypergeometric function is defined as the convergent power series
	\begin{align}\label{Confluent-HG}
		\Phi(a;c;z):={_1F_1}(a;c;z)=\sum_{j=0}^\infty\frac{(a)_j}{(c)_j}\frac{z^j}{j!}, \quad z\in\mathbb{D},
	\end{align}
	where $(x)_j$ is the Pochhammer symbol defined in \eqref{Pochhammer-Symbol}. 
\end{definition}
The function $\Phi(a;c;z)$ is analytic in $\mathbb{C}$ and satisfies the Kummer's differential equation 
\begin{align*}
	zw''(z)+(c-z)w'(z)-aw(z)=0.
\end{align*}
If we replace $b$ by $1/\varrho$ and $z$ by $z\varrho$ in the series \eqref{Gaussian-HG} and allow $\varrho\to0$, we obtain the series \eqref{Confluent-HG}.
   Below, we mention certain well-known properties of $\Phi(a;c;z)$ given by \eqref{Confluent-HG}.
   \begin{align}\label{Gamma-Function-Rep-CHGF}
   	\Phi(a;c;z)
   	=\frac{\Gamma(c)}{\Gamma(a)}\sum_{j=0}^\infty\frac{\Gamma(a+j)}{\Gamma(c+j)}\frac{z^j}{j!},
   \end{align}
 \begin{align}\label{Derivative-Property-CHGF}
 	\Phi'(a;c;z)=\frac{a}{c}\Phi(a+1;c+1;z),
 \end{align}
and 
 \begin{align}\label{Integral-Rep-CHGF}
 	\Phi(a;c;z)=
 	\frac{\Gamma(c)}{\Gamma(a)\Gamma(c-a)}\int_{0}^1{t^{a-1}(1-t)^{c-a-1}e^{tz}}\,dt, \quad (\mathrm{Re}\,c>\mathrm{Re}\,a>0).
 \end{align}
%%%%==============================================%%%%
Further, the function 
\begin{align*}
	z\Phi(a;c;z)=z{_1F_1}(a;c;z)=z+\sum_{j=2}^\infty\frac{(a)_{j-1}}{(c)_{j-1}}\,\frac{z^j}{(j-1)!}, \quad z\in\mathbb{D},
\end{align*}
is the {\it normalised} (shifted) confluent hypergeometric function. The following result related to the starlikeness of $z\Phi(a;c;z)$ will be used to prove \Cref{Thrm-CHGF1}.
\begin{lemma}[{Miller and Mocanu \cite[p. 236]{Miller-Mocanu-Book-2000-Diff-Sub}}]\label{Lemma-Corollary-MilMocanu}
	If $c\leq 1+N(a-1)$, where  
	\begin{align*}
		N(a)=
		\begin{cases}
			|a|+\frac{1}{2} \quad\text{if}\quad |a|\geq\frac{1}{3},\\
			\frac{3(a)^2}{2}+\frac{2}{3} \quad\text{if}\quad |a|\leq\frac{1}{3},			
		\end{cases}
	\end{align*}
then $z\Phi(a;c;z)$ is starlike in $\mathbb{D}$.	
\end{lemma}
%%%%%%%%%%=========%%%%%%%%%%%%%%%
\section{Main Results}
%%%%%%%%%%=========%%%%%%%%%%%%%%%
By making use of \Cref{Lemma-3.4h-p132-Miller-Mocanu}, \Cref{Lemma-Kustner-2007-HGF-JMAA},  and the above mentioned properties of Gaussian hypergeometric function, we prove \Cref{Thrm-LemB-Impl-Neph-GHGF} and \Cref{Thrm-GHGF2}.
%%%==========T1=========%%%
\begin{theorem}\label{Thrm-LemB-Impl-Neph-GHGF}
	Let $p:\mathbb{D}\to\mathbb{C}$ be analytic and satisfies $p(0)=1$. Let $\varphi_{\scriptscriptstyle{L}}(z):=\sqrt{1+z}$ and
	\begin{align*}
		p(z)+\beta zp'(z)\prec\varphi_{\scriptscriptstyle{L}}(z),\qquad \beta>0.
	\end{align*}
	If $\beta\geq\beta_L\approx0.158379$, then $p(z)\prec\varphi_{\scriptscriptstyle{Ne}}(z)$, where $\beta_L$ is the unique root of
	\begin{align}\label{Eq-BetaL}
		\frac{3}{\Gamma(-\frac{1}{2})}\sum_{j=0}^\infty\frac{\Gamma(-\frac{1}{2}+j)}{j!\,(1+j\beta)}-1 = 0.
	\end{align}
	The estimate on $\beta$ is best possible.
\end{theorem}
%%%====================%%%
\begin{proof}
	An elementary analysis shows that the analytic function
	\begin{align}\label{Def-qbeta-Integral-Form-Lem}
		\Psi_\beta(z)=\frac{1}{\beta}\int_0^1\frac{t^{\frac{1}{\beta}-1}}{(1+zt)^{-1/2}}\,dt, \qquad \beta>0
	\end{align}
	is a solution of the first-order linear differential equation
	\begin{align*}
		\Psi_\beta(z)+\beta z\Psi'_\beta(z)=\varphi_{\scriptscriptstyle{L}}(z).
	\end{align*}  
	In view of the representation \eqref{Integral-Rep-Gaussian-HGF} of the Gaussian hypergeometric function, it is easy to see that the function $\Psi_\beta(z)$ given by \eqref{Def-qbeta-Integral-Form-Lem} has the form
	\begin{align}\label{qbeta-GHF-Form}
		\Psi_\beta(z)=F\left(-\frac{1}{2},\frac{1}{\beta};\frac{1}{\beta}+1;-z\right).
	\end{align}
For brevity, we now split the proof into two steps.\\
{\bf Step I}. {\it In this step, we prove that $p(z)+\beta zp'(z)\prec\varphi_{\scriptscriptstyle{L}}(z)$ implies $p(z)\prec \Psi_\beta(z)$, $\beta>0$}.
\par  
	For $\xi\in\mathbb{C}$, define $\vartheta(\xi)=\xi$ and $\lambda(\xi)=\beta$ so that
	\begin{align*}
	\Theta(z)=z\Psi'_\beta(z)\lambda(\Psi_\beta(z))=\beta z\Psi'_\beta(z)=
	\beta zF'\left(-\frac{1}{2},\frac{1}{\beta};\frac{1}{\beta}+1;-z\right).
	\end{align*}
	This, on using the identity \eqref{Derivative-Property-GHGF}, gives
	\begin{align}\label{Def-qbeta-GHGF-Form-Lem}
		\Theta(z)=\frac{\beta}{2(1+\beta)} zF\left(\frac{1}{2},\frac{1}{\beta}+1;\frac{1}{\beta}+2;-z\right).
	\end{align}
	We prove that the function $\Theta(z)$ given by \eqref{Def-qbeta-GHGF-Form-Lem} is starlike in $\mathbb{D}$ by showing that $\sigma(\Theta)\geq0$, where $\sigma(\cdot)$ is defined in \eqref{Order-ST}. 
	%%%%%%%%%%%%%%%%%%%%%%%%%%%%%%%%%%%%%%%%%%%%%%%%%
	For the normalized hypergeometric function on the right side of \eqref{Def-qbeta-GHGF-Form-Lem}, we have $a=1/2,b=1/\beta+1$ and $c=1/\beta+2$, so that the condition $0<a\leq b\leq c$ easily holds. Therefore, by \Cref{Lemma-Kustner-2007-HGF-JMAA}, it follows that
	\begin{align*}
		\sigma\left(zF(a,b;c;z)\right)
		\geq 1-\frac{ab}{b+c}
		%&=1-\frac{\frac{1}{2}\left(\frac{1}{\beta}+1\right)}{\left(\frac{1}{\beta}+1\right)+\left(\frac{1}{\beta}+2\right)}\\
		=1-\frac{1+\beta}{2\left(2+3\beta\right)}
		=\frac{3+5\beta}{2\left(2+3\beta\right)}
		>0 
		\qquad (\because\beta>0).	                                
	\end{align*}
%	%\begin{align*}
%	%\sigma\left(zF(a,b;c;z)\right)
%	%&\geq 1-\frac{ab}{b+c}\\ 
%	%&=1-\frac{\frac{1}{2}\left(\frac{1}{\beta}+1\right)}{\left(\frac{1}{\beta}+1\right)+\left(\frac{1}{\beta}+2\right)}\\
%	%&=1-\frac{1+\beta}{2\left(2+3\beta\right)}\\
%	%&=\frac{3+5\beta}{2\left(2+3\beta\right)}\\
%	%&>0 \qquad \qquad \qquad \qquad (\because\beta>0).	                                
%	%\end{align*}
	This shows that the hypergeometric function $zF\left(\frac{1}{2},\frac{1}{\beta}+1;\frac{1}{\beta}+2;-z\right)$ is starlike in $\mathbb{D}$, thereby proving the starlikeness of $\Theta(z)$ defined in \eqref{Def-qbeta-GHGF-Form-Lem}.	  
	%%%%%%%%%%%%%%%%%%%%%%%%%%%%%%%%%%%%%%%%%%%%%%%%%%%%%
	Since $\beta>0$ and $\Theta(z)$ is starlike, the function $$h(z)=\vartheta\left(\Psi_\beta(z)\right)+\Theta(z)=\Psi_\beta(z)+\Theta(z)$$ 
	satisfies
	\begin{align*}
		\mathrm{Re}\left(\frac{zh'(z)}{\Theta(z)}\right)= \mathrm{Re}\left(\frac{1}{\beta}+\frac{z\Theta'(z)}{\Theta(z)}\right)>0.
	\end{align*}
	In view of \Cref{Lemma-3.4h-p132-Miller-Mocanu}, we conclude that the differential subordination 
	$$p(z)+\beta zp'(z)\prec \Psi_\beta(z)+\beta z\Psi'_\beta(z)=\varphi_{\scriptscriptstyle{L}}(z)$$ 
	implies the subordination
	$p(z)\prec \Psi_\beta(z)$, where $\Psi_\beta(z)$ is given by \eqref{Def-qbeta-Integral-Form-Lem} (or \eqref{qbeta-GHF-Form}).
	\par
		Now the desired subordination $p(z)\prec\varphi_{\scriptscriptstyle{Ne}}(z)$ will hold true if the subordination $\Psi_\beta(z)\prec\varphi_{\scriptscriptstyle{Ne}}(z)$ holds.\\ 
{\bf Step II}. {\it In this step, we prove that $\Psi_\beta(z)\prec\varphi_{\scriptscriptstyle{Ne}}(z)$ if, and only if, $\beta\geq\beta_L\approx0.158379$}.
	\par 
%%	{\bf Claim.}
%%	The necessary and sufficient condition for the subordination $\Psi_\beta\prec\varphi_{\scriptscriptstyle{Ne}}$ to hold is that $\beta\geq\beta_L\approx0.158379$.
%%	\par 
	{\it Necessity.}
	Let $\Psi_\beta(z)\prec\varphi_{\scriptscriptstyle{Ne}}(z)$, $z\in\mathbb{D}$. Then 
	\begin{align}\label{NS-DSI-GHGFN}
		\varphi_{\scriptscriptstyle {Ne}}(-1)<\Psi_\beta(-1)<\Psi_\beta(1)<\varphi_{\scriptscriptstyle {Ne}}(1).
	\end{align}
	On using the representation \eqref{qbeta-GHF-Form} and the identity \eqref{Gamma-Function-Rep-Gaussian-HGF}, the condition \eqref{NS-DSI-GHGFN} yields the two inequalities
	\begin{align*}
		\frac{1}{3}\leq F\left(-\frac{1}{2},\frac{1}{\beta};\frac{1}{\beta}+1;1\right)=\frac{1}{\Gamma(-\frac{1}{2})}\sum_{j=0}^\infty\frac{\Gamma(-\frac{1}{2}+j)}{j!\,(1+j\beta)}
	\end{align*}
	and
	\begin{align*}
		\frac{5}{3} \geq F\left(-\frac{1}{2},\frac{1}{\beta};\frac{1}{\beta}+1;-1\right)=\frac{1}{\Gamma(-\frac{1}{2})}\sum_{j=0}^\infty\frac{\Gamma(-\frac{1}{2}+j)}{j!\,(1+j\beta)}(-1)^j.
	\end{align*}
	Or equivalently,
	\begin{align*}
		\tau(\beta):= \frac{1}{\Gamma(-\frac{1}{2})}\sum_{j=0}^\infty\frac{\Gamma(-\frac{1}{2}+j)}{j!\,(1+j\beta)}-\frac{1}{3} \geq 0
	\end{align*}
	and
	\begin{align*}
		\delta(\beta):=\frac{5}{3} - \frac{1}{\Gamma(-\frac{1}{2})}
		\sum_{j=0}^\infty\frac{\Gamma(-\frac{1}{2}+j)}{j!\,(1+j\beta)}(-1)^j \geq 0.
	\end{align*}
	%%%%%%%%%%%%%%% NEW %%%%%%%%%%%%%%%%%%%%%%%%%
	A computer based numerical computation shows that for $\beta\in(0,\infty)$, 
	\begin{align*}
		\tau(\beta)\in\left(-\frac{1}{3},\frac{2}{3}\right)
		\quad \text{ and } \quad
		\delta(\beta)\in\left(\frac{5}{3}-\sqrt{2},\;\frac{2}{3}\right).
	\end{align*}
	That is, as $\beta$ varies from $0$ to $\infty$, $\delta(\beta)$ is positive, while $\tau(\beta)$ takes positive as well as negative values. Moreover, 
	$$\tau'(\beta)=\frac{-1}{\Gamma(-\frac{1}{2})}
	\sum_{j=1}^\infty\frac{\Gamma(-\frac{1}{2}+j)}{(j-1)!\,(1+j\beta)^2}$$
	takes values from the interval $(0,\infty)$ for each $\beta\in(0,\infty)$. This shows that $\tau(\beta)$ is strictly increasing in $(0,\infty)$.  Therefore, both conditions $\tau(\beta)\geq0$ and $\delta(\beta)\geq0$ hold true for $\beta\geq\beta_L\approx0.158379$, where $\beta_L$ is the unique root of $\tau(\beta)$. See the plots of $\tau(\beta)$ and $\delta(\beta)$ in \Cref{Plots-TauDel-HGF-SRN}.
	%%%%%%%%% FIGURES-Tau-Delta %%%%%%%%%%%%% %%%%%%%%%%%%%%%%%%%%%%%%%%%%%%%%%%%%%%%%%
	\begin{figure}[H]
		\begin{subfigure}{0.45\textwidth}
%			\centering
			\includegraphics[height=2.0in,width=2.5in]{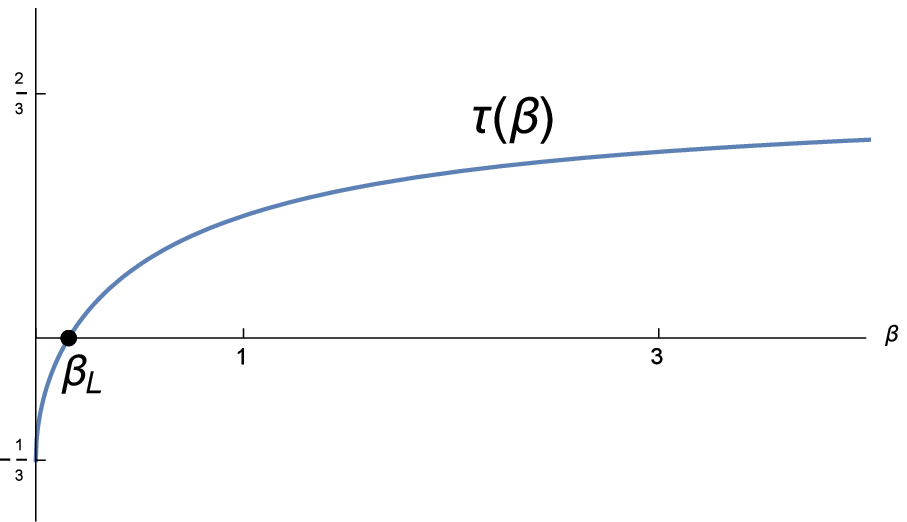}
		\end{subfigure}
		\begin{subfigure}{0.45\textwidth}
			\centering
			\includegraphics[height=2.0in,width=2.5in]{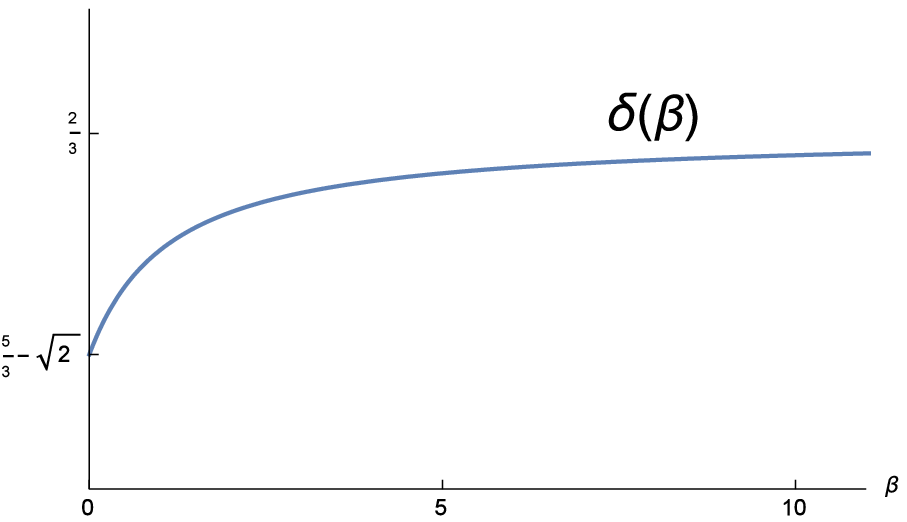}
		\end{subfigure}
		\caption{Plots of $\tau(\beta)$ and $\delta(\beta)$, $\beta>0$.}
		\label{Plots-TauDel-HGF-SRN}
	\end{figure}
	%%%%%%%%%%%%%%%%%%%%%%%%%%%%%%
	\par  
	{\it Sufficiency.}
	Since the function $\varphi_{\scriptscriptstyle{Ne}}(z)$ is univalent in $\mathbb{D}$ and $\Psi_\beta(0)=\varphi_{\scriptscriptstyle{Ne}}(0)=1$, it is sufficient to prove that $\Psi_\beta(\mathbb{D})\subset\varphi_{\scriptscriptstyle{Ne}}(\mathbb{D})$ for $\beta\geq\beta_L$.\\
	The square of the distance from the point (1,0) to the points on the nephroid curve \eqref{Equation-of-Nephroid} is 
	\begin{align*}
		d_1(\theta):=\frac{16}{9}-\frac{4 \cos^2\theta}{3}, \quad 0\leq\theta<2\pi,
	\end{align*}
	and the square of the distance from (1,0) to the points on the curve
	\begin{align*}
	\Psi_\beta(e^{i\theta})=&F\left(-\frac{1}{2},\frac{1}{\beta};\frac{1}{\beta}+1;-e^{i\theta}\right)\\
	                    =&\frac{1}{\Gamma(-\frac{1}{2})}\sum_{j=0}^\infty\frac{(-1)^j\Gamma(-\frac{1}{2}+j)}{j!\,(1+j\beta)}(\cos{j\theta}+i\sin{j\theta}), \quad 0\leq\theta<2\pi,
	\end{align*} 
	is given by
	\begin{align*}
		d_2(\theta,\beta):=\left(\frac{1}{\Gamma\left(-\frac{1}{2}\right)}
		\sum_{j=0}^{\infty}
		C(j,\beta)
		\cos(j\theta)-1\right)^2
		+\left(\frac{1}{\Gamma\left(-\frac{1}{2}\right)}
		\sum_{j=0}^{\infty}
		C(j,\beta)
		\sin(j\theta)\right)^2,
	\end{align*}
	where
	\begin{align*}
		C(j,\beta):=\frac{(-1)^j\Gamma(-\frac{1}{2}+j)}{j!\,(1+j\beta)}.
	\end{align*}
	Since the curves $\varphi_{\scriptscriptstyle{Ne}}(e^{i\theta})$ and $\Psi_\beta(e^{i\theta})$ are symmetric about the real axis, we may choose $\theta\in[0,\pi]$. Now, the difference of square of the distances from the point (1, 0) to the points on the boundary curves $\varphi_{\scriptscriptstyle{Ne}}(e^{i\theta})$ and $\Psi_\beta(e^{i\theta})$, respectively, is
	\begin{align*}
		d(\theta,\beta):=&d_1(\theta)-d_2(\theta,\beta)\\
		         =&\frac{16}{9}-\frac{4 \cos^2\theta}{3}-
		        \left(
		        \frac{\sum_{j=0}^{\infty}C(j,\beta)\cos(j\theta)}{\Gamma\left(-\frac{1}{2}\right)}
%		         \sum_{j=0}^{\infty}C(j,\beta)\cos(j\theta)
		         -1\right)^2
		         -\left(
		         \frac{\sum_{j=0}^{\infty}C(j,\beta)\sin(j\theta)}{\Gamma\left(-\frac{1}{2}\right)}
%		         \sum_{j=0}^{\infty}C(j,\beta)\sin(j\theta)
		         \right)^2,		         
	\end{align*}
  A computer based numerical computation shows that $d(\theta,\beta)\geq0$ for each $\theta\in[0,\pi]$ whenever $\beta\geq\beta_L\approx0.158379$ and $d(\theta,\beta)<0$ for some $\theta\in(\pi-\epsilon,\pi)$, $\epsilon\to0$ whenever $\beta<0.158379$, see \Cref{Table-Numeric-Values-H}. This shows that the region bounded by the curve $\Psi_\beta(e^{i\theta})$ is completely contained in $\varphi_{\scriptscriptstyle{Ne}}(\mathbb{D})$ whenever $\beta\geq\beta_L$. 
  Moreover, the estimate on $\beta$ is best possible as
   \begin{align*}
   	d(\pi,\beta_L)
   	=&\frac{4}{9}-
   	\left(\frac{1}{\Gamma\left(-\frac{1}{2}\right)}
   	\sum_{j=0}^{\infty} (-1)^jC(j,\beta_L)-1\right)^2\\
   	=&\frac{4}{9}-
   	\left(\frac{1}{\Gamma\left(-\frac{1}{2}\right)}
   	\sum_{j=0}^{\infty}\frac{\Gamma(-\frac{1}{2}+j)}{j!\,(1+j\beta_L)} -1\right)^2\\
   	=&\frac{4}{9}-
   	\left(\left(\tau(\beta_L)+\frac{1}{3}\right)-1\right)^2\\
   	=&\frac{4}{9}-
   	\left(\frac{1}{3}-1\right)^2 \qquad\qquad \left(\because \tau(\beta_L)=0\right)\\
   	=&0	         
   \end{align*}
%%%-----------------Numerical-Table-------------%%%
  See \Cref{Figure-HGT1-I} for the graphical illustration of the above proved facts and the sharpness of the bound $\beta_L$ for the containment $\Psi_\beta(\mathbb{D})\subset\varphi_{\scriptscriptstyle{Ne}}(\mathbb{D})$. This proves the sufficiency of $\beta\geq\beta_L$ for the subordination $\Psi_\beta\prec\varphi_{\scriptscriptstyle{Ne}}$.\\
  The desired result now follows by combining the conclusions of Step I and Step II.
	%%%%%%%%%%%%%%%%%%%%%%%%%%%%%%%%%%%%%%%%%%%%%%%%%%% 
\end{proof}
%%%-----------------Numerical-Table-------------%%%
%%%%%%%%%%%%%%%%% TABLE-Numeric-Values %%%%%%%%%%%%%%%%%%%%%%
\begin{center}
	\begin{table}[H]
		\centering
		\begin{tabular}{ |c|c|c| }
			\hline
			{\boldmath{$\theta$}}  & {\boldmath{$d(\theta,\beta)$, $\beta=\beta_L\approx0.158379$}}  & {\boldmath{$d(\theta,\beta)$, $\beta=0.1583737<\beta_L$}}\\
			\hline
			$3$          & $0.0893992$              & $0.0893943$\\
			\hline
			$3.14$       & $0.000230464$            & $0.000223834$\\
			\hline 
			$3.141$      & $0.0000596419$           & $0.0000530052$\\
			\hline
			$3.1415$     & $9.83806\times10^{-6}$   & $3.19942\times10^{-6}$\\
			\hline
			$3.14159$    & $6.4166\times10^{-6}$    & $-2.22177\times10^{-7}$\\
			\hline
			$3.141592$   & $6.40162\times10^{-6}$   & $-2.37156\times10^{-7}$\\
			\hline
			$3.1415926$  & $6.39958\times10^{-6}$   & $-2.39199\times10^{-7}$\\
			\hline
			$3.14159265$ & $6.39953\times10^{-6}$   & $-2.39247\times10^{-7}$\\
%%			\hline 
%%			$3.141592653$          & $0$             & $0$\\
%%			\hline 
%%			$3.1415926535$          & $0$             & $0$\\
			\hline 
			$\pi$        & $0$                      & $-2.39248\times10^{-7}$\\
			\hline 
		\end{tabular}
		\caption{Numerical computations}
		\label{Table-Numeric-Values-H}
	\end{table}
\end{center}
%%%%%%%%%%%%%%%%%%%%%%%%%%%%%%%%%%%%%%%%%%%%%%% 
%%%%%%%%%%%%%%%%%%%%%%%%%%%%%%%%%%%%%%%%%%%%%%%%%%%%%%%%%%%%%%%%
%%%%%%%%%% FIGURES-Tau-Delta %%%%%%%%%%%%% 
%%\begin{figure}[H]
%%	\begin{subfigure}{0.45\textwidth}
%%		\centering
%%		\includegraphics[scale=1.0]{Figure-Tau-Beta}
%%	\end{subfigure}
%%	\begin{subfigure}{0.45\textwidth}
%%		\centering
%%		\includegraphics[scale=1.0]{Figure-Delta-Beta}
%%	\end{subfigure}
%%	\caption{Plots of $\tau(\beta)$ and $\delta(\beta)$, $\beta>0$.}
%%	\label{Plots-TauDel-HGF-SRN}
%%\end{figure}
%%%%%%%%%%%%%%%%%%%%%%%%%%%%%%%%%%%%%%%%%%%%%%%%%%%%%%%%%%%%%%%%
\begin{figure}[h]
		\includegraphics[scale=.7]{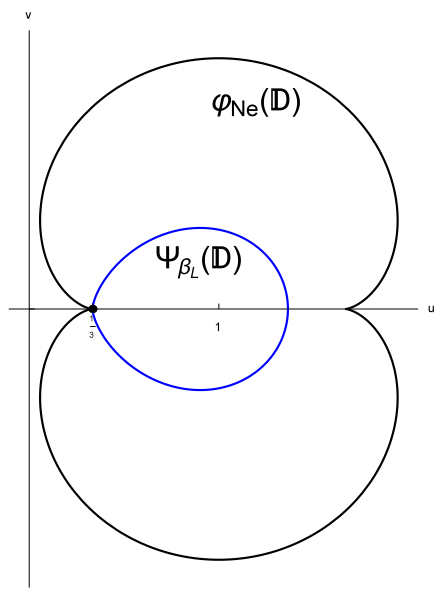}
		\caption{$\Psi_{\beta_L}(\mathbb{D})\subset\varphi_{\scriptscriptstyle{Ne}}(\mathbb{D})$}
		\label{Figure-HGT1-I}
\end{figure}
%%%%%%%%%%%%%%%%%%%%%%%%%%%%%%%%%%%%
The following sufficient condition for the function class $\mathcal{S}^*_{Ne}$ is a direct application of \Cref{Thrm-LemB-Impl-Neph-GHGF} obtained by setting $p(z)=zf'(z)/f(z)$.

\begin{corollary}
	Let $f\in\mathcal{A}$, and let
	\begin{align}\label{Definition-J}
		\mathcal{G}(z):=1-\frac{zf'(z)}{f(z)}+\frac{zf''(z)}{f'(z)},\quad z\in\mathbb{D}.
	\end{align}
%	Let $\mathcal{G}(z)$ be defined as in \eqref{Definition-J}, and let
     If the function $f(z)$ satisfies
	\begin{align*}
		\left(1+\beta\,\mathcal{G}(z)\right)\frac{zf'(z)}{f(z)}\prec \varphi_{\scriptscriptstyle{L}}(z),
	\end{align*}
	then $f\in\mathcal{S}^*_{Ne}$ for $\beta\geq\beta_L$, where $\beta_L$ is the unique root of \eqref{Eq-BetaL}.
\end{corollary}
%%%==================================
\begin{theorem}\label{Thrm-GHGF2}
	Let the analytic $p$ satisfies $p(0)=1$ and let
	\begin{align*}
	p(z)+\beta zp'(z)\prec 1+z, \qquad \beta>0.
	\end{align*}
	Then $p\prec\varphi_{\scriptscriptstyle{Ne}}$ whenever $\beta\geq1/2$ and this estimate on $\beta$ is sharp.
\end{theorem}

\begin{proof}
	Consider the first-order linear differential equation 
	$$q_\beta(z)+\beta zq'_\beta(z)=1+z$$ 
	whose analytic solution is the function $q_\beta(z)$ given by
	\begin{align*}
		q_\beta(z)=\frac{1}{\beta}\int_0^1 t^{\frac{1}{\beta}-1}(1+zt)\,dt=F\left(-1,\frac{1}{\beta};\frac{1}{\beta}+1;-z\right),\quad z\in\mathbb{D}.
	\end{align*}
	as its solution. Defining the functions $\vartheta$ and $\lambda$ as in \Cref{Thrm-LemB-Impl-Neph-GHGF} we obtain 
	\begin{align*}
		\Theta(z)=zq'_\beta(z)\lambda(q_\beta(z))=\beta zq'_\beta(z)
		%=\beta zF'\left(-1,\frac{1}{\beta};\frac{1}{\beta}+1;-z\right)
		=\frac{\beta}{1+\beta} zF\left(0,\frac{1}{\beta}+1;\frac{1}{\beta}+2;-z\right)
		=\frac{\beta}{1+\beta} z,
	\end{align*}
	which is clearly a starlike function in $\mathbb{D}$. Also, $$h(z)=\vartheta\left(q_\beta(z)\right)+\Theta(z)=q_\beta(z)+\Theta(z)$$ 
	satisfies
	$\mathrm{Re}\left({zh'(z)}/{\Theta(z)}\right)>0$ in $\mathbb{D}$. Therefore, it follows from \Cref{Lemma-3.4h-p132-Miller-Mocanu} that 
	$$p(z)+\beta zp'(z)\prec 1+z=q_\beta(z)+\beta zq'_\beta(z)$$ 
	implies $p\prec {q_\beta}$.
	\par 
	To get $p\prec\varphi_{\scriptscriptstyle{Ne}}$, it now remains to prove that $q_\beta\prec\varphi_{\scriptscriptstyle{Ne}}$.\\ 
	If $q_\beta\prec\varphi_{\scriptscriptstyle{Ne}}$, then
	$\varphi_{\scriptscriptstyle {Ne}}(-1)<q_\beta(-1)<q_\beta(1)<\varphi_{\scriptscriptstyle {Ne}}(1)$ which is  
 is equivalent to 
%%%the conditions
	\begin{align*}
		F\left(-1,\frac{1}{\beta};\frac{1}{\beta}+1;1\right)-\frac{1}{3}\geq0 \;
		\quad \text{or}, \quad \beta\geq \frac{1}{2}
	\end{align*}
	and
	\begin{align*}
		\frac{5}{3} - F\left(-1,\frac{1}{\beta};\frac{1}{\beta}+1;-1\right)\geq0 \; 
		\quad \text{or}, \quad	
		\beta\geq -\frac{5}{2}
	\end{align*}
	Therefore the necessary condition for $q_\beta\prec\varphi_{\scriptscriptstyle{Ne}}$ is that $\beta\geq\max\{1/2,-5/2\}=1/2$.\\
	As in \Cref{Thrm-LemB-Impl-Neph-GHGF}, it can be easily verified that whenever $\beta\geq1/2$, the distance $d_1(\theta)$ from (1,0) to the points on the nephroid curve $\varphi_{\scriptscriptstyle{Ne}}(e^{i\theta})$ is always greater than or equal to the distance $d_2(\theta,\beta)$ from (1,0) to the points on the curve $q_\beta(e^{i\theta})$, $0\leq\theta<2\pi$. This shows that $q_\beta(\mathbb{D})\subset\varphi_{\scriptscriptstyle{Ne}}(\mathbb{D})$ whenever $\beta\geq1/2$. Hence $\beta\geq1/2$ is also sufficient for the subordination $q_\beta\prec\varphi_{\scriptscriptstyle{Ne}}$ to hold true. Moreover,
	\begin{align*}
		d_1(0)=d_2(0,1/2)
		\quad \text{and} \quad
		d_1(\pi)=d_2(\pi,1/2).
	\end{align*}
	Therefore the estimate on $\beta$ can not be decreased further.
\end{proof}
\begin{corollary}
	Let $\mathcal{G}(z)$ be given by \eqref{Definition-J}. If $f\in\mathcal{A}$ satisfies the subordination
	\begin{align*}
		\left(1+\beta\,\mathcal{G}(z)\right)\frac{zf'(z)}{f(z)}\prec 1+z,
	\end{align*}
	then $f\in\mathcal{S}^*_{Ne}$ whenever $\beta\geq1/2$.
\end{corollary}
%%%%%%%%%%%%%%%%%%%%%%%%%%%%%%%%%%%%%%%%%%%%%%%%%%%%
In the following theorem, we make use of \Cref{Lemma-3.4h-p132-Miller-Mocanu}, \Cref{Lemma-Corollary-MilMocanu}, and the properties of the confluent hypergeometric function $\Phi(a;c;z)$ defined in \eqref{Confluent-HG}.
%%%===========================%%%
\begin{theorem}\label{Thrm-CHGF1}
	Let $p$ be analytic in $\mathbb{D}$ satisfying $p(0)=1$. For $\varphi_{\scriptscriptstyle{e}}(z):=e^z$, let the differential subordination
	\begin{align*}
		p(z)+\beta zp'(z)\prec \varphi_{\scriptscriptstyle{e}}(z), \qquad \beta>0.
	\end{align*}
	holds. Then $p(z)\prec\varphi_{\scriptscriptstyle{Ne}}(z)$ whenever $\beta\geq\beta_e\approx1.14016$, where $\beta_e$ is the unique solution of  
	\begin{align*}
		\sum_{j=0}^\infty\frac{1}{j!\,(1+j\beta)}-\frac{5}{3}=0.
	\end{align*}	
	The estimate on $\beta$ can not be improved further.
\end{theorem}
%%%=============
\begin{proof}
	Consider the function
	\begin{align}\label{Def-Sol-CHGF}
		\psi_\beta(z)=\frac{1}{\beta}\int_0^1\
		e^{zt}\,t^{\frac{1}{\beta}-1}\,dt, \qquad \beta>0.
	\end{align}
	It can be easily verified that $\psi_\beta(z)$ given by \eqref{Def-Sol-CHGF} is an analytic solution of the  linear differential equation
	\begin{align*}
		\psi_\beta(z)+\beta z\psi'_\beta(z)=\varphi_{\scriptscriptstyle{e}}(z).
	\end{align*}  
	Using the representation \eqref{Integral-Rep-CHGF} of the confluent hypergeometric function, it can be observed that
	\begin{align}\label{qbeta-CGHF-Form}
		\psi_\beta(z)=\Phi\left(\frac{1}{\beta};\frac{1}{\beta}+1;z\right).
	\end{align}
	Define $\vartheta(\xi)=\xi$ and $\lambda(\xi)=\beta$ so that 
	\begin{align*}
		\Theta(z)=z\psi'_\beta(z)\lambda(\psi_\beta(z))=\beta z\psi'_\beta(z)=
		\beta z\Phi'\left(\frac{1}{\beta};\frac{1}{\beta}+1;z\right),
	\end{align*}
	which upon using \eqref{Derivative-Property-CHGF} yields
	\begin{align}\label{Def-qbeta-CHGF-Form-Lem}
		\Theta(z)=\frac{\beta}{(1+\beta)} z\Phi\left(\frac{1}{\beta}+1;\frac{1}{\beta}+2;z\right).
	\end{align}
  We now use \Cref{Lemma-Corollary-MilMocanu} to prove that $\Theta(z)$ given by \eqref{Def-qbeta-CHGF-Form-Lem} is starlike in $\mathbb{D}$. Here $a=1/\beta+1$ and $c=1/\beta+2$ so that  $c-1=1/\beta+1$ and 
  	\begin{align*}
  	N(a-1)=
  	\begin{cases}
  		\frac{1}{\beta}+\frac{1}{2} \quad\text{if}\quad \frac{1}{\beta}\geq\frac{1}{3},\\
  		\frac{3}{2\beta^2}+\frac{2}{3} \quad\text{if}\quad \frac{1}{\beta}\leq\frac{1}{3}.			
  	\end{cases}
  \end{align*}
  In both the cases, the inequality $c-1\geq N(a-1)$ holds. Therefore, $\Theta(z)$ is starlike and consequently the function 
	$$h(z)=\vartheta\left(\psi_\beta(z)\right)+\Theta(z)=\psi_\beta(z)+\Theta(z)$$ 
	satisfies $\mathrm{Re}\left({zh'(z)}/{\Theta(z)}\right)>0$ for each $z\in\mathbb{D}$.
	In view of \Cref{Lemma-3.4h-p132-Miller-Mocanu}, we conclude that the differential subordination 
	 $p(z)+\beta zp'(z)\prec \psi_\beta(z)+\beta z\psi'_\beta(z)=\varphi_{\scriptscriptstyle{e}}(z)$
%%%	$$p(z)+\beta zp'(z)\prec \psi_\beta(z)+\beta z\psi'_\beta(z)=\varphi_{\scriptscriptstyle{e}}(z)$$ 
	implies $p(z)\prec \psi_\beta(z)$, where $\psi_\beta(z)$ is given by \eqref{qbeta-CGHF-Form}.
	\par
	Now the subordination $p(z)\prec\varphi_{\scriptscriptstyle{Ne}}(z)$ will be attained if  $\psi_\beta(z)\prec\varphi_{\scriptscriptstyle{Ne}}(z)$.\\ 
		{\bf Claim.}
		The necessary and sufficient condition for the subordination $\psi_\beta\prec\varphi_{\scriptscriptstyle{Ne}}$ to hold true is that $\beta\geq\beta_e\approx1.14016$.
		\par 
	{\it Necessity.}
	Let $\psi_\beta(z)\prec\varphi_{\scriptscriptstyle{Ne}}(z)$, $z\in\mathbb{D}$. Then 
	\begin{align}\label{NS-DSI-CHGFN}
		\varphi_{\scriptscriptstyle {Ne}}(-1)<\psi_\beta(-1)<\psi_\beta(1)<\varphi_{\scriptscriptstyle {Ne}}(1).
	\end{align}
	On using \eqref{qbeta-CGHF-Form} and \eqref{Gamma-Function-Rep-CHGF} in \eqref{NS-DSI-CHGFN}, we obtain the two inequalities
	\begin{align*}
		\mu(\beta):=
		\sum_{j=0}^\infty\frac{(-1)^j}{j!\,(1+j\beta)}-
		\frac{1}{3}
		\geq0
	\end{align*}
	and
	\begin{align*}
		\rho(\beta):=
		\frac{5}{3} 
		-\sum_{j=0}^\infty\frac{1}{j!\,(1+j\beta)}
		\geq 0.
	\end{align*}
 We note that:
   \begin{align*}
   	\lim_{\beta\searrow0}\mu(\beta)=\frac{1}{e}-\frac{1}{3}>0, \qquad 
   	\lim_{\beta\nearrow\infty}\mu(\beta)=\frac{2}{3}>0
   \end{align*}
 and
 \begin{align*}
 	\lim_{\beta\searrow0}\rho(\beta)=\frac{5}{3}-e<0, \qquad 
 	\lim_{\beta\nearrow\infty}\rho(\beta)=\frac{2}{3}>0.
 \end{align*}
 Moreover, a computation shows that both $\mu(\beta)$ and $\rho(\beta)$ are strictly increasing functions of $\beta\in(0,\infty)$. Therefore, the inequalities $\mu(\beta)\geq0$ and $\rho(\beta)\geq0$ are true for $\beta\geq\beta_e$, where $\beta_e\approx1.14016$ is the unique root of $\rho(\beta)=0$. 
 See the plots of $\mu(\beta)$ and $\rho(\beta)$ in \Cref{Plot-mu,Plot-rho}.
%%%%%%%%%%%%%%%%%%%%%%%%%%%%%
\par
 {\it Sufficiency.} In order to attain the subordination $\psi_\beta(z)\prec\varphi_{\scriptscriptstyle{Ne}}(z)$, we only need to prove that $\psi_\beta(\mathbb{D})\subset\varphi_{\scriptscriptstyle{Ne}}(\mathbb{D})$ whenever $\beta\geq\beta_e$. Likewise in \Cref{Thrm-LemB-Impl-Neph-GHGF}, the difference of the square of the distances from the point (1, 0) to the points on the boundary curves $\varphi_{\scriptscriptstyle{Ne}}(e^{i\theta})$ and $\psi_\beta(e^{i\theta})$, respectively, is
	\begin{align*}
		d(\theta,\beta):=&d_1(\theta)-d_2(\theta,\beta)\\
		=&\frac{4}{3}\left(\frac{4}{3}-{\cos^2\theta}\right)-
		\left(\sum_{j=0}^{\infty}\frac{\cos(j\theta)}{{j!\,(1+j\beta)}}-1\right)^2
		-\left(\sum_{j=0}^{\infty}\frac{\sin(j\theta)}{{j!\,(1+j\beta)}}\right)^2.		         
	\end{align*}  
Since both the curves $\varphi_{\scriptscriptstyle{Ne}}(e^{i\theta})$ and $\psi_\beta(e^{i\theta})$ are symmetric about the real line, we restrict $\theta$ to $[0,\pi]$. A computation shows that $d(\theta,\beta)\geq0$ for each $\theta\in[0,\pi]$ whenever $\beta\geq\beta_e\approx1.14016$ and $d(\theta,\beta)<0$ for some $\theta$ whenever $\beta<1.14016$. This shows that the region bounded by the curve $\psi_\beta(e^{i\theta})$ lies in the interior of  $\varphi_{\scriptscriptstyle{Ne}}(\mathbb{D})$ whenever $\beta\geq\beta_e$.
Further, the estimate on $\beta$ can not be improved as
\begin{align*}
	d(0,\beta_e)
	=&\frac{4}{9}-
	\left(\sum_{j=0}^{\infty}\frac{1}{{j!\,(1+j\beta_e)}}-1\right)^2\\
	=&\frac{4}{9}-
	\left(\left(\frac{5}{3}-\rho(\beta_e)\right)-1\right)^2\\
	=&\frac{4}{9}-
	\left(\frac{5}{3}-1\right)^2 \qquad\qquad \left(\because \rho(\beta_e)=0\right)\\
	=&0	         
\end{align*}
See \Cref{Figure-CHGFT} for the geometrical interpretation of the sharpness of $\beta_e$.
%%%%%%%%%%%%%%%%%%%%%%%%%%%%%%%%%%%%%%%%%%%%%%%%%%% 
\end{proof}
%%%%%%%%%% FIGURES-Tau-Delta %%%%%%%%%%%%% %%%%%%%%%%%%%%%%%%%%%%%
\begin{figure}[H]
	\begin{subfigure}[b]{0.3\textwidth}
		\centering
		\includegraphics[height=1.5in,width=1.65in]{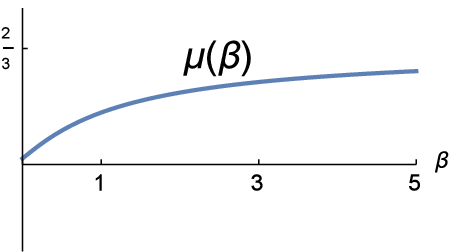}
		\caption{Plot of $\mu(\beta)$, $\beta>0$.}
		\label{Plot-mu}
	\end{subfigure}
	\begin{subfigure}[b]{0.3\textwidth}
		\centering
		\includegraphics[height=1.5in,width=1.65in]{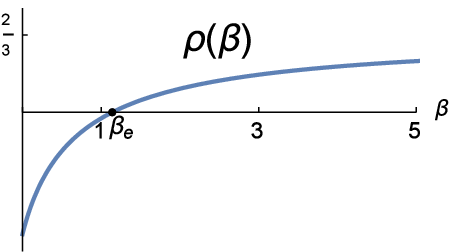}
		\caption{Plot of $\rho(\beta)$, $\beta>0$.}
		\label{Plot-rho}
	\end{subfigure}
	\begin{subfigure}[b]{0.3\textwidth}
		\centering
		\includegraphics[scale=0.92]{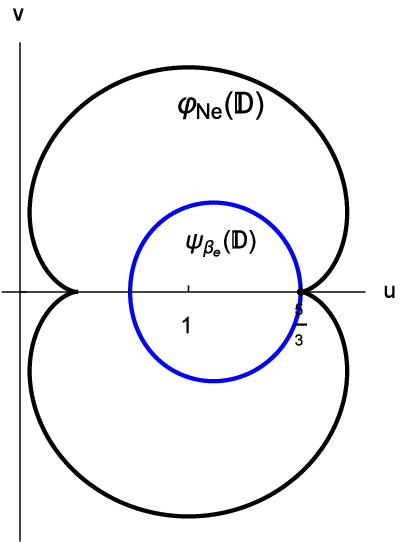}
		\caption{$\psi_{\beta_e}(\mathbb{D})\subset\varphi_{\scriptscriptstyle{Ne}}(\mathbb{D})$}
		\label{Figure-CHGFT}
	\end{subfigure}
	\caption{}
\end{figure}
%%%height=2.0in,width=1.7in
%%%==================================%%
The following sufficient condition for the nephroid starlikeness of $f\in\mathcal{A}$ is obtained on setting $p(z)={zf'(z)}/{f(z)}$ in \Cref{Thrm-CHGF1}.
\begin{corollary}
	Let $\mathcal{G}(z)$ be defined as in \eqref{Definition-J}.
	If $f\in\mathcal{A}$ satisfies the subordination
	\begin{align*}
		\left(1+\beta\,\mathcal{G}(z)\right)\frac{zf'(z)}{f(z)}\prec \varphi_{\scriptscriptstyle{e}}(z),
	\end{align*}
	then $f\in\mathcal{S}^*_{Ne}$ for $\beta\geq\beta_e\approx1.14016$.
\end{corollary}
%%%===============================%%%
%%%%%%%%%%%%%%%%%%%%%%%%%%%%%%%%%%%%%%
\section{Conclusion}
%%%%%%%%%%%%%%%%%%%%%%%%%%%%%%%%%%%%%%
In this work, certain differential subordination-implication problems have been discussed. We have employed a new technique to solve the problem by applying the well-known properties of Gaussian and confluent (or Kummer) hypergeometric functions. In addition, analytic clarification of certain set-inclusions have been supplied in this paper which the earlier authors have claimed to be true without providing any details. All the results proved yield sufficient conditions for the nephroid starlikeness of a normalized analytic function. For a quick overview, we summarize the subordination implications discussed in this paper in \Cref{Table:II}. 
\par 
In view of the recent papers of Srivastava et al.
\cite{HMS-q-Hankel-2019-MPaid,
	HMS-q-Coeff-2019-HMJ,	
	HMS-q-ConicD-2019-Rocky, 
	HMS-q-Janowski-2019-Filomat,  
	HMS-q-Hankel-2021-BSM},
we remark that the works related to the nephroid domain carried out by the authors of this manuscript has several future prospects for $q$-extension.  
%%%%%%%%%%%%%%%%%%%%%%%%%%%%%%%%%%%%%%%%%%%%%%%%%%%%%%% 
%%%=============================================
\begin{center}
	\begin{longtable}{ |M{4cm}|c|c|M{2.1cm}|M{2.4cm}| }
		\hline
		{\bf Differential Subordination} & {\boldmath{$\mathcal{P}(z)$}} & {\bf Implication} & {\boldmath$\beta$} & {\bf Sharp/ non-sharp \boldmath$\beta$}\\
		\hline 
		\multirow{3}{10em}{$p(z)+\beta{zp'(z)}\prec\mathcal{P}(z)$} & $\sqrt{1+z}$  & \multirow{3}{6em}{$p(z)\prec\varphi_{\scriptscriptstyle{Ne}}(z)$ $=1+z-{z^3}/{3}$} &  $0.158379$ &\multirow{3}{2em}{sharp}\\
		\cline{2-2}  \cline{4-4}    & $1+z$        &      & $1/2$ &\\
		\cline{2-2}  \cline{4-4}    & $e^z$        &      & $1.14016$   &\\
		\hline
		\caption{Subordination implications studied in this paper}
		\label{Table:II}
	\end{longtable}
\end{center}
%%%%%%%%%%%%%%%%%%%%%%%%%%%%%%%%%%%%%%%%%%%%%%%
\vspace{-1.0em}

{\bf Acknowledgment.} This work of the authors were supported by the Project No. CRG/2019/000200/MS of Science and Engineering Research Board, Department of Science and Technology, New Delhi, India.
%%%======================================================

%%%%%%%%%%%%%%%%%% BIBLIOGRAPHY %%%%%%%%%%%%%%%%%%%%%%%%%%%%

\end{document}